\documentclass{amsart}
\usepackage{amsmath} 
\usepackage{amsfonts} 
\usepackage{amssymb} 
\usepackage[utf8]{inputenc} 
\usepackage[british]{babel} 
\usepackage{hyperref} 
\usepackage[protrusion=true,expansion=false]{microtype} 
\usepackage{mathtools}
\usepackage{tabularx}
\usepackage{amsthm} 
\usepackage{tikz}
\usetikzlibrary{positioning}
\usetikzlibrary{arrows.meta}
\usetikzlibrary{decorations.pathmorphing}
\usepackage{yfonts}
\usepackage{tikz-cd}
\usepackage{float}
\usepackage{lmodern}
\usepackage{accents}

\definecolor{dblue}{RGB}{0,0,178}
\hypersetup{
	unicode=true,
	colorlinks=true,
	citecolor=dblue,
	linkcolor=dblue,
	anchorcolor=dblue,
	urlcolor=dblue
}

\newcommand*{\defeq}{\mathrel{\vcenter{\baselineskip0.5ex \lineskiplimit0pt
                     \hbox{\scriptsize.}\hbox{\scriptsize.}}}%
                     =}
\newcommand{\calL}{\mathcal{L}}
\newcommand{\N}{\mathbb{N}}
\newcommand{\Z}{\mathbb{Z}}
\newcommand{\abs}[1]{\left\vert#1\right\vert}
\newcommand{\lra}{\longleftrightarrow}
\newcommand{\e}{\equiv}
\newcommand{\gen}[1]{\langle#1\rangle}
\newcommand{\vphi}{\varphi}
\newcommand{\ca}{{\circledast}}
\newcommand{\deq}{\mathrel{\dot{=}}}
\newcommand{\ttA}{{\tt A}}
\newcommand{\ttS}{{\tt S}}
\newcommand{\ttF}{{\tt F}}
\newcommand{\ttV}{{\tt V}}
\newcommand{\ttP}{{\tt P}}
\newcommand{\ttE}{{\tt E}}
\newcommand{\ttM}{{\tt M}}
\newcommand{\axiom}[1]{\mathsf{#1}}
\newcommand{\NF}{\axiom{NF}}
\newcommand{\ZF}{\axiom{ZF}}
\newcommand{\PA}{\axiom{PA}}

\definecolor{dcolor}{rgb}{0,0,0.7}
\definecolor{ecolor}{rgb}{0.7,0,0}

\newcommand{\distinguish}[1]{{\color{dcolor}\widetilde{#1}}}
\newcommand{\exclude}[1]{{\color{ecolor}\undertilde{#1}}}

\newtheoremstyle{carriagedefn}{}{}{\itshape}{}{\bfseries}{.}{\newline}{}

\newtheorem{thm}{Theorem}[section] 
\newtheorem{prop}[thm]{Proposition} 
\newtheorem*{mainthm}{Main Theorem}

\theoremstyle{carriagedefn}
\newtheorem{defn}[thm]{Definition} 

\theoremstyle{remark}
\newtheorem*{eg}{Example}

\title{Stratifiable formulae are not context-free}
\author{Calliope Ryan-Smith}
\email{c.Ryan-Smith@leeds.ac.uk}
\urladdr{https://academic.calliope.mx}
\address{School of Mathematics, University of Leeds, LS2 9JT, UK}
\date{\today}
\thanks{The author's work was financially supported by EPSRC via the Mathematical Sciences Doctoral Training Partnership [EP/W523860/1]. For the purpose of open access, the author has applied a Creative Commons Attribution (CC BY) licence to any Author Accepted Manuscript version arising from this submission. No data are associated with this article.}

\begin{document}


\begin{abstract}
Stratified formulae were introduced by Quine as an alternative way to attack Russell's Paradox. Instead of limiting comprehension by size (as in \(\ZF\) set theory, using its axiom scheme of separation), unlimited comprehension is given to formulae that are in some sense descended from formulae of typed set theory. By keeping variables in a stratified structure, the most common candidates for inconsistency such as \(\{x\mid x\notin x\}\) are eliminated. Under the usual syntax of set theory, the set of stratified formulae form a formal language. We show that, unlike the full class of well-formed formulae of set theory, this language is not context-free, and extend the result to its complement. Therefore, much like the axioms of \(\PA\) and \(\ZF\) (under their usual axiomatizations), the theory \(\NF\) as a formal language is not context-free. We then introduce a non-standard syntax of set theory and show that with this syntax there is a restricted class of formulae, the \emph{exo-stratified} formulae, that is context-free and full (up to relabelling of variables).
\end{abstract}

\maketitle

\section{Introduction}\label{s:introduction}

The issue of Russell's Paradox is central to the development of modern set theory and was overcome in \cite{whitehead_principia_1992} through the use of the theory of types. Well-formed formulae in this theory need to obey typing restrictions that caused the structure of models to form a natural `layered' system. However, this system of set theory had issues identified by Quine in \cite{quine_new_1937} such as losing the boolean algebra structure of the universe that one may hope for, as negations of formulae only took complements within their own type. Instead, Quine described \emph{stratified} formulae inspired by the theory of types. A stratified formula is a well-formed formula of typed set theory that has `forgotten' its typing. From this, one can produce the set theory now called \emph{New Foundations}:
\begin{equation*}
\begin{aligned}
\NF=\{&(\forall x,y)(x=y\lra(\forall z)(z\in x\lra z\in y))\}\cup\\
\{&(\forall \bar{x})(\exists y)(\forall z)(z\in y\lra\vphi(z,\bar{x}))\mid\vphi\text{ is stratified}\}.
\end{aligned}
\end{equation*}
The description of stratified formulae is quite simple in the context of the typed set theory, but it is not nearly so simple as the language of all well-formed formulae of set theory.
\begin{mainthm}
The formal language of stratified formulae of set theory is not context-free.
\end{mainthm}
Therefore the usual axiomatisation of \(\NF\) is also not context-free, much like the usual axiomatisations of \(\ZF\) and \(\PA\).\footnote{Though one should note well that Hailperin's finite axiomatisation of \(\NF\) from \cite{hailperin_set_1944} (or indeed any finite axiomatisation) is not only context-free, but regular.} However, by using a somewhat more enriched syntax of set theory via the introduction of more variables, one can recover a context-free set of stratified formulae. This set, the \emph{exo-stratified} formulae, are indeed full in the sense that for every stratified formula $\vphi$, there is an exo-stratfied formula that is equal to \(\vphi\), modulo a renaming of variables.\footnote{In fact, this syntax was first described in \cite[p.~62]{holmes_proof_2021} in order to demonstrate that the language of type theory does not presuppose knowledge of arithmetic, and so quite naturally fits into the realm of stratified formulae.}

\subsection{Structure of the paper}\label{s:introduction;ss:structure}

In Section~\ref{s:formal-languages} we give a brief introduction to formal languages, including any preliminaries. In particular, context-free languages and their notational conventions are defined. These languages have closure properties, such as an analogy to the pumping lemma for context-free languages. For example, the generalized Ogden's lemma, from \cite{bader_generalization_1982}, is key to our goal (see Theorem~\ref{thm:g-ogdens}).

In Section~\ref{s:syntax-of-sets} we give an overview of a common formal syntax of set theory. We will demonstrate that this syntax is in the format of a context-free language, and so the entire language of first-order formulae of set theory is context-free. From this, we pick out the subset of \emph{stratified} formulae, which may be considered the well-formed formulae of Typed Set Theory under the operation of `forgetting about the types'.

In Section~\ref{s:main-thm} we appeal to the generalized Ogden's lemma to show that the language of stratified formulae has insufficient closure to be context-free. We then give an overview of how to modify the proof to show that the same result applies to many similar languages, such as the set of stratified formulae in a logical syntax with different basic logical connectives, and the language of acyclic formulae from \cite{al-johar_axiom_2014}.

\section{Formal languages and context-free languages}\label{s:formal-languages}

By an \emph{alphabet} we mean a finite set of symbols, and by a \emph{formal language over $\Sigma$}, we mean a subset of $\Sigma^{{<}\omega}$. We refer to the elements of an alphabet as \emph{symbols} and elements of $\Sigma^{{<}\omega}$ as \emph{strings}. When referring to arbitrary strings, Greek letters ($\alpha,\beta,\dots$) will be used. Given a string $\alpha\colon\{0,\dots,n-1\}\to\Sigma$, we denote its \emph{length} by $\abs{\alpha}\defeq n$. Given strings $\alpha$ and $\beta$, we denote by $\alpha\beta$ their concatenation. As we are working with languages of formulae, we will take additional care to avoid ambiguities. In particular, equality of abstract strings will be denoted \(\e\), and the symbol \(\e\) will not appear in alphabets. Similarly, we avoid using the symbols \(=\), \((\), and \()\) in alphabets, instead opting for \(\deq\), \([\), and \(]\).

\emph{Context-free languages} were introduced in \cite{chomsky_three_1956} as part of a hierarchy of computability of formal languages, and describe many very natural language structures in mathematics, including the set of all formulae of set theory. They took their place as `type 2' languages, being more complex than type 3 languages (regular languages), but less complex than type 1 languages (context-sensitive languages). We shall define context-free languages using \emph{context-free grammars}, though there is a common equivalent definition involving \emph{pushdown automata} that will not be used here. A more thorough treatment of context-free languages can be found in \cite{chiswell_course_2009}.

\begin{eg}
In general we present context-free languages through grids of symbol replacements. Given an alphabet $\Sigma$, we will build a language $\calL\subseteq\Sigma^{{<}\omega}$, but will require an additional alphabet of \emph{nonterminal symbols} $\Omega$ disjoint from $\Sigma$, and a \emph{start symbol} $\ttS\in\Omega$. We then have a collection of \emph{production rules} that tell us what we can do with nonterminal symbols. For example, letting $\Sigma=\{0,1\}$ and $\Omega=\{\ttS\}$, we could have the replacement rule ``$\ttS\mapsto 0\ttS1$''.

In this case, we would start with $\ttS$ and be able to replace it with $0\ttS1$. Then we would again be able to replace the $\ttS$ with another $0\ttS1$, giving us $00\ttS11$, and so on:
\begin{equation*}
\ttS\implies 0\ttS1\implies00\ttS11\implies000\ttS111\implies\cdots
\end{equation*}
However, we will never be able to obtain a string with no nonterminal symbols in it, so we will never finish producing. If instead we add a second rule ``$\ttS\mapsto 01$'' then we can now break out of this cycle at any point to produce strings in our language.
\begin{equation*}
\begin{tikzcd}
\ttS\ar[d,Rightarrow]\ar[r,Rightarrow]&0\ttS1\ar[d,Rightarrow]\ar[r,Rightarrow]&00\ttS11\ar[d,Rightarrow]\ar[r,Rightarrow]&\cdots\\
01&0011&000111&\ddots
\end{tikzcd}
\end{equation*}
When presenting the production rules, we group those derived from the same nonterminal symbol together, so we could present this context-free grammar as
\begin{equation*}
\begin{tabular}{ccccc}
$\ttS$&$\mapsto$&$0\ttS1$&$|$&$01$
\end{tabular}
\end{equation*}
which generates the language $\{0^n1^n\mid n\in\N,n>0\}$. Note that we are denoting production rules with the arrow \(\mapsto\), while we are denoting constructions based on these rules with the arrow \(\implies\).

However, we needn't only have one nonterminal symbol. For example, let $\Sigma=\{0,1,+,\deq\}$ and $\Omega=\{\ttS,\ttF\}$. Then the grammar
\begin{equation*}
\begin{tabular}{ccccccc}
$\ttS$&$\mapsto$&$\ttF\deq\ttF$&&&&\\
$\ttF$&$\mapsto$&$\ttF+\ttF$&$|$&$0$&$|$&$1$
\end{tabular}
\end{equation*}
will generate strings of the form
\begin{equation*}
\delta_0+\delta_1+\cdots+\delta_n\deq\varepsilon_0+\varepsilon_1+\cdots+\varepsilon_m,
\end{equation*}
where each $\delta_i,\varepsilon_j\in\{0,1\}$. Let us now formalise this definition.
\end{eg}

\begin{defn}
A \emph{context-free grammar} is a quadruple $G=(\Omega,\Sigma,R,\ttS)$, where:
\begin{enumerate}
\item[\textup{1.}] $\Omega$ and $\Sigma$ are disjoint alphabets. $\Omega$ defines the \emph{nonterminal symbols}, and $\Sigma$ the \emph{terminal symbols}. The language generated by such a grammar will be over alphabet $\Sigma$. We will always denote nonterminal symbols as capital letters of the {\tt teletype font};
\item[\textup{2.}] $R$ is a collection of \emph{production rules}, a finite subset of $\Omega\times(\Omega\cup\Sigma)^{{<}\omega}$; and
\item[\textup{3.}] $\ttS\in\Omega$ is the \emph{start symbol}.
\end{enumerate}

The production rules for a context-free grammar give rise to the binary `one step' relation on $(\Omega\cup\Sigma)^{{<}\omega}$ that we denote $\implies$ (think of this like $\beta$-reduction), where $\alpha_0\ttA\alpha_1\implies\alpha_0\beta\alpha_1$ if \(\alpha_0,\alpha_1,\beta\in(\Omega\cup\Sigma)^{{<}\omega}\), \(\ttA\in\Omega\), and $\gen{\ttA,\beta}\in R$.

We say that $\alpha\in\Sigma^{{<}\omega}$ is \emph{generated} by the context-free grammar $G$ if there is a finite sequence $\alpha_0,\alpha_1,\dots,\alpha_{n-1}$ with $\alpha_i\implies\alpha_{i+1}$ for all $i$, $\alpha_0\e\ttS$, and $\alpha_{n-1}\e\alpha$. Similarly, the \emph{language} generated by the context-free grammar $G$ is the set of all strings generated by $G$. A \emph{context-free language} is the collection of all strings generated by some context-free grammar.
\end{defn}

As we saw in our examples, context-free grammars are not typically presented with formal quadruples. Suppose that $\Omega={\{\ttA_0,\ttA_1,\dots,\ttA_{n-1}\}}$, and that
\begin{equation*}
R=\{\gen{\ttA_i,\alpha_{i,j}}\mid 0\leq i<n,0\leq j\leq m_i\}.
\end{equation*}
Then we write
\begin{equation*}
\begin{tabular}{cccccccc}
$\ttA_0$&$\mapsto$&$\alpha_{0,0}$&$|$&$\alpha_{0,1}$&$|$&$\cdots$&$\alpha_{0,m_0}$\\
$\ttA_1$&$\mapsto$&$\alpha_{1,0}$&$|$&$\alpha_{1,1}$&$|$&$\cdots$&$\alpha_{1,m_1}$\\
&&&&$\vdots$&&&\\
$\ttA_{n-1}$&$\mapsto$&$\alpha_{n-1,0}$&$|$&$\alpha_{n-1,1}$&$|$&$\cdots$&$\alpha_{n-1,m_{n-1}}$.
\end{tabular}
\end{equation*}
Note that not all nonterminal symbols need to have associated production rules, which can be implemented by allowing \(m_i=-1\) in the case that \(\ttA_i\) is unproductive.

While still being able to describe a wide range of useful formal languages, context-free languages are still only a small part of the world of formal languages (indeed, there are only countably many of them), and so they have strong restrictions that lead to natural closure properties. We are able to exploit a particular family of closure properties, `pumping lemmas', in order to show that certain languages are not context-free. These lemmas generally say that if you have a large enough string in a context-free language then, since the description of the language is only finite, you must produce a loop when you generate it. Hence, you can take a part of the string and pump it to generate more strings in the language.

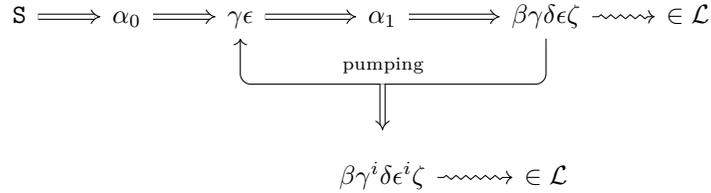
\begin{figure}[H]
\begin{tikzcd}
\ttS\ar[Rightarrow,r]&
\cdots\ar[Rightarrow,r]&
\gamma\varepsilon\ar[Rightarrow,r]&
\cdots\ar[Rightarrow,r]&
\makebox[\widthof{$\beta\gamma\delta\varepsilon\zeta$}][l]{$\mathrlap{\beta\gamma\delta\varepsilon\zeta\in\calL}$}\ar[ll,phantom,bend left=27,"{\scriptsize\text{pumping}}"] \ar[ll,to path={-- ([yshift=-4ex]\tikztostart.south) -| (\tikztotarget.south)},rounded corners]\\
&&&\ar[Rightarrow,d,yshift=1.9ex]\\
&&&\makebox[\widthof{$\beta\gamma^i\delta\varepsilon^i\zeta$}][l]{$\mathrlap{\beta\gamma^i\delta\varepsilon^i\zeta\in\calL}$}
\end{tikzcd}
\centering
\caption{An oversimplification of the pumping lemma for context-free languages.}
\end{figure}
The most familiar pumping lemma in the family is also the easiest to understand, exemplifying the class. Simply put, for all context-free languages there is a critical point $n$ (depending on the language) after which we can pump any string from the language in some fashion.
\begin{thm}[Pumping lemma for context-free languages]\label{thm:pumping}
Let $\calL$ be a context-free language. There is a number $n\geq1$ such that, for all $\alpha\in\calL$, if $\lvert\alpha\rvert>n$, then there is a partition of $\alpha$ into substrings $\alpha\e\beta\gamma\delta\varepsilon\zeta$ such that:
\begin{enumerate}
\item[\textup{1.}] $\lvert\gamma\varepsilon\rvert\geq1$;
\item[\textup{2.}] $\lvert\gamma\delta\varepsilon\rvert\leq n$; and
\item[\textup{3.}] for all $i\in\N$, $\beta\gamma^i\delta\varepsilon^i\zeta\in\calL$.
\end{enumerate}
\end{thm}
\begin{proof}
Immediate corollary of Theorem~\ref{thm:g-ogdens}.
\end{proof}
However, this will not be powerful enough for our purposes, so we will need the \emph{generalized Ogden's lemma}, published in \cite{bader_generalization_1982}. The generalized Ogden's lemma allows one to more finely control the pumping by distinguishing or excluding certain symbols in every string, so that the pumping must include at least one distinguished symbol and may not include any excluded symbols.
\begin{thm}[Generalized Ogden's lemma]\label{thm:g-ogdens}
Let $\calL$ be a context-free language. There is a number $n\geq1$ such that, for all $\alpha\in\calL$ and all colorings of the symbols of $\alpha$ as either `excluded', `unlabeled', or `distinguished', if there are more than $n^{e+1}$ distinguished positions, where $e$ is the number of excluded positions, then there is a partition of $\alpha$ into substrings $\alpha\e \beta\gamma\delta\varepsilon\zeta$ such that:
\begin{enumerate}
\item[\textup{1.}] $\gamma\varepsilon$ contains at least one distinguished position and no excluded positions;
\item[\textup{2.}] if $x$ is the number of distinguished positions in $\gamma\delta\varepsilon$, and $y$ the number of excluded positions in $\gamma\delta\varepsilon$, then $x\leq n^{y+1}$; and
\item[\textup{3.}] for all $i\in\N$, $\beta\gamma^i\delta\varepsilon^i\zeta\in\calL$.
\end{enumerate}
\end{thm}

\begin{proof}
See \cite[p.~405]{bader_generalization_1982}.
\end{proof}

\section{Syntax of formulae in set theory}\label{s:syntax-of-sets}

Throughout, our language of set theory will involve the nine symbols $[$, $]$, $\to$, $\forall$, $\deq$, $\in$, $w$, $\ca$, and $\bot$. The set of well-formed formulae of set theory are defined by the following rules:
\begin{enumerate}
\item[\textup{1.}] \emph{Variables}. $w$ is a variable. If $\alpha$ is a variable, then ${\alpha}\,{\ca}$ is a variable.\footnote{Usually one would use $w$, $w'$, $w''$, and so on. In this case, we are using $\ca$ so that it is clearer when the symbol is raised to a power (compare $'^{n}$ with $\ca^n$). An alternative convention denotes the variable $w$ with $n$ apostrophes as $w^{(n)}$. We will not use this convention here, as we wish to refer to these $\ca$ symbols in some proofs while not necessarily referring to the associated $w$.}
\item[\textup{2.}] \emph{Atomic formulae}. $\bot$ is an atomic formula. If $\alpha$ and $\beta$ are variables, then $\alpha\deq\beta$ and $\alpha\in\beta$ are atomic formulae.
\item[\textup{3.}] \emph{Formulae}. Every atomic formula is a formula. If $\vphi$ and $\psi$ are formulae, and $\alpha$ is a variable, then $[\forall \alpha]\vphi$ and $[\vphi\to\psi]$ are formulae.
\end{enumerate}
We skip defining terms (in the first-order sense of the word) as terms coincide with variables in our relational language.

\begin{prop}
The set of well-formed formulae of set theory form a context-free language.
\end{prop}

\begin{proof}
The definition is already in the form of a context-free grammar. Writing {\tt F} for `formula', {\tt A} for `atomic formula', and {\tt V} for `variable', the grammar exhibited by Equation~(\ref{cfg:usual-set-theory}) here generates that language.
\begin{equation}\label{cfg:usual-set-theory}
\begin{tabular}{ccccccc}
$\ttF$&$\mapsto$&$\ttA$&$|$&$[\ttF\to\ttF]$&$|$&${[{\forall}\,{\ttV}]}\,{\ttF}$\\
$\ttA$&$\mapsto$&$\bot$&$|$&$\ttV\deq\ttV$&$|$&$\ttV\in\ttV$\\
$\ttV$&$\mapsto$&$w$&$|$&${\ttV}\,{\ca}$&
\end{tabular}
\end{equation}\qedhere
\end{proof}

\begin{defn}[Stratified formula]
Let $\vphi$ be a well-formed formula of set theory, and let $X$ be the set of variables that appear in $\vphi$. A \emph{stratification} of $\vphi$ is a function $\sigma\colon X\to\Z$ such that:
\begin{enumerate}
\item Whenever $x\in y$ appears as a subformula of $\vphi$, $\sigma(x)=\sigma(y)-1$; and
\item whenever $x\deq y$ appears as a subformula of $\vphi$, $\sigma(x)=\sigma(y)$.
\end{enumerate}
A formula is \emph{stratified} if there exists a stratification of it.
\end{defn}

Note that one occasionally makes a distinction between \emph{stratified} formulae and \emph{stratifiable} formulae in the literature, in which a stratified formula is understood to exist with a stratification witnessing that it is stratified, while a stratifiable formula has no such witness picked out. We do not make use of this distinction.

\section{Main theorem}\label{s:main-thm}

\begin{thm}\label{thm:main-thm}
The set of well-formed stratified formulae of set theory do not form a context-free language.
\end{thm}

\begin{proof}
We appeal to the generalized Ogden's lemma, showing that for all $n\geq1$ we can produce a stratified formula $\vphi_n$ and coloring of the symbols such that when it is sufficiently pumped we fail to have a stratified formula. This formula will be
\begin{equation*}
\vphi_n\e[^{M+1}w\in w{\distinguish{\ca}}^{N}\beta^M\to w{\exclude{\ca}}^{N+N!}\in w]
\end{equation*}
Where $\beta$ is the string $\to w{\distinguish{\deq}}w]$, $N=n+1$, and $M=n^{N+N!+1}$. From this string we distinguish all $\ca$ symbols appearing in $w\ca^N$ and all $\deq$ symbols appearing in any $\beta$. We exclude all $\ca$ symbols appearing in $w\ca^{N+N!}$.\footnote{As a convention, distinguished positions are written with an $\distinguish{\text{overtilde}}$ and the excluded positions are written with an $\exclude{\text{undertilde}}$.} In this case, the number of distinguished positions is $N+M=N+n^{N+N!+1}$ and the number of excluded positions is $N+N!$. Hence $d>n^{e+1}$.

Henceforth, for any substring $\alpha$ of $\vphi_n$, let $d(\alpha)$ be the number of distinguished positions appearing in $\alpha$, and $e(\alpha)$ the number of excluded positions. Suppose, for a contradiction, that $\vphi_n$ may be partitioned as $\beta\gamma\delta\varepsilon\zeta$ as in the generalized Ogden's lemma, so that:
\begin{enumerate}
\item[\textup{1.}] $d(\gamma\varepsilon)\geq1$;
\item[\textup{2.}] $d(\gamma\delta\varepsilon)\leq n^{e(\gamma\delta\varepsilon)+1}$; and
\item[\textup{3.}] for all $i\in\N$, $\beta\gamma^i\delta\varepsilon^i\zeta$ is a stratified formula.
\end{enumerate}
We shall derive a contradiction from this by showing that we must have $\gamma\delta\varepsilon$ be some subset of the $\ca$ symbols coming from ${w}\,{\ca^{N}}$, and therefore for some $i$ we have that $\beta\gamma^i\varepsilon\delta^i\zeta$ contains both $w\in{{w}\,{\ca^{N+N!}}}$ and ${{w}\,{\ca^{N+N!}}}\in w$, failing stratifiability. The strategy for showing this is as follows:
\begin{enumerate}
\item[\textup{1.}] First we show that no $[$ may appear in $\gamma\varepsilon$.
\item[\textup{2.}] By well-nestedness, no $]$ may appear in $\gamma\varepsilon$ either.
\item[\textup{3.}] Therefore the number of atomic formulae appearing in $\beta\gamma^i\delta\varepsilon^i\zeta$ must remain the same.
\item[\textup{4.}] Therefore $\varepsilon$ and $\gamma$ may only consist of $\ca$s.
\item[\textup{5.}] We show that $\gamma\delta\varepsilon$ spans only the distinguished $\ca$ symbols from $w\ca^{N}$.
\item[\textup{6.}] Therefore there is $i$ such that $\beta\gamma^i\delta\varepsilon^i\zeta$ contains both $w\in{{w}\,{\ca^{N+N!}}}$ and ${{w}\,{\ca^{N+N!}}}\in w$ as subformulae, contradicting stratifiability.
\end{enumerate}
Now for the promised detail. Suppose that some $[$ appears in $\gamma\varepsilon$. Since both \(\beta\delta\zeta\) and \(\beta\gamma\delta\varepsilon\zeta\) are stratified formulae, and hence well-nested, $\gamma\varepsilon$ must contain the same number of $]$ symbols as it does $[$. Since all $[$ symbols occur at the beginning of the formula, $\gamma\delta\varepsilon$ spans an initial segment (minus some initial $[$ symbols). Since all $]$ symbols lie after the $N$ distinguished $\ca$ symbols from $w\in{{w}\,{\ca^N}}$, we have that $d(\gamma\delta\varepsilon)\geq N=n+1>n^1$, so for $d(\gamma\delta\varepsilon)\leq n^{e(\gamma\delta\varepsilon)+1}$ to hold, \(\gamma\delta\varepsilon\) needs to contain some excluded positions. Therefore, $\gamma\delta\varepsilon$ spans from an initial $[$ all the way to the term $w\ca^{N+N!}\in w$, which in particular means that it contains \emph{all} $N+M$ distinguished positions appearing in $\vphi_n$. Since there are only $N+N!$ excluded positions in all of $\vphi_n$, there is no way for $d(\gamma\delta\varepsilon)=N+M=N+n^{N!+N+1}$ to be at most $n^{e(\gamma\delta\varepsilon)+1}$. Hence $\gamma\varepsilon$ contains no $[$ symbols.

Now, by the well-nestedness of stratified formulae, $\gamma\varepsilon$ also contains no $]$ symbols. We may also deduce that the number of atomic formulae appearing in $\beta\gamma^i\delta\varepsilon^i\zeta$ remains the same for all $i\in\N$, as the only way to increase the number of atomic formulae is the introduction of new $\to$ symbols, which bring with them new $[$ and $]$ symbols.

Since no new atomic formulae are being introduced, we get no new copies of any relation symbols ($\deq$ and $\in$), and no new variables (so no new $w$ symbols). Also, there are no $\forall$ or $\bot$ symbols in $\vphi_n$, so there is no chance of $\gamma\varepsilon$ containing those. Therefore, $\gamma$ and $\varepsilon$ consist only of some number of $\ca$ symbols. By $e(\gamma\varepsilon)=0$, $\gamma$ and $\varepsilon$ can only contain $\ca$ symbols from the distinguished collection in $w\in w\ca^N$, and must contain at least one.

Therefore, letting $k$ be the number of $\ca$s contained in $\gamma\varepsilon$, note that $1\leq k\leq N$ and so $k$ divides $N!$. Then
\begin{equation*}
\begin{aligned}
\beta\gamma^{1+N!/k}\delta\varepsilon^{1+N!/k}\zeta&\e[^{M+1}w\in {w\ca^{N-k+(1+N!/k)k}}\beta^M\to {w\ca^{N+N!}}\in w]\\
&\e[^{M+1}w\in {w\ca^{N+N!}}\beta^M\to {w\ca^{N+N!}}\in w]
\end{aligned}
\end{equation*}
which is not stratified. Hence we conclude that the language of stratified formulae is not context-free.
\end{proof}

The counterexample used in the proof of Theorem~\ref{thm:main-thm} may be easily modified to show that a wide variety of related languages are not context-free. For example, we could easily replace each $\to$ with a $\land$, so systems of first-order syntax that use $\lnot$ and $\land$ instead of $\bot$ and $\to$ also don't generate a context-free set of stratified formulae. Similarly the same can be shown for the syntax generated with the single connective of \emph{alternative denial} $(\vphi\mid\psi)$, meaning $\lnot(\vphi\land\psi)$, as described in \cite{quine_new_1937}.

A related language is the language of \emph{acyclic formulae}, introduced in \cite{al-johar_axiom_2014}. A formula \(\vphi\) is \emph{acyclic} if:
\begin{enumerate}
\item[\textup{1.}] For each variable \(x\) appearing in \(\vphi\), either all occurrences of \(x\) are free or all occurrences of \(x\) are bound by the same quantifier.
\item[\textup{2.}] There is no sequence of variables \(\gen{x_i\mid0\leq i\leq n}\) and sequence of atomic formulae \(\gen{\alpha_i\mid 0\leq i<n}\) appearing as distinct subformulae of \(\vphi\) such that \(x_0\e x_n\) and, for all \(i<n\), \(\alpha_i\) is of the form \(x_i\deq x_{i+1}\), \(x_{i+1}\deq x_i\), \(x_i\in x_{i+1}\), or \(x_{i+1}\in x_i\). We allow any \(n\geq 1\), so in particular \(x\deq x\) is not acyclic, for example.
\end{enumerate}
Note that any acyclic formula must be stratified. One can alter the formula \(\vphi_n\) from the proof of Theorem~\ref{thm:main-thm} to be acyclic and work with the same argument. In the formula below, let \(\prod_{i=1}^M(\alpha_i)\) refer to the concatenation \(\alpha_1\alpha_2\cdots\alpha_M\). Let
\begin{equation*}
\psi_n\e[^{M+1}w\in{w\distinguish{\ca}^N}\prod_{\substack{i=1\\i\neq N,N+N!}}^{M+2}\left(\to w\distinguish{\deq}{w\ca^i}]\right)\to w\exclude{\ca}^{N+N!}\in w],
\end{equation*}
where \(N=n+1\) and \(M=n^{N+N!+1}\). The rest of the proof is almost the same, with the exception of noting that \(\gamma\varepsilon\) needn't only contain some of the distinguished \(\ca\) symbols, but may indeed contain some of the \(\ca\) from the substrings \(\to{{w}\mathrel{\distinguish{\deq}}{{w}\,{\ca^i}}}\). However, each of \(\gamma\) and \(\varepsilon\) must still be contained to some of the \(\ca\) symbols from a single variable substring, at least one of which is included in the distinguished substring \({w}\,{\distinguish{\ca}^N}\). Hence, after pumping an appropriate number of times, \(\beta\gamma^i\delta\varepsilon^i\zeta\) is no longer acyclic (indeed, not even stratified).

Finally, the set of \emph{unstratified} formulae of set theory is also not context-free. To see this, we use a slight modification on the same counterexample.
\begin{equation*}
\vphi_n'\e [^{M+1}w\in {{w}\,{\distinguish{\ca}}^N}\beta^M\to {{w}\,{\exclude{\ca}}^N}\in w]
\end{equation*}
Here $N=n+1$, $M=n^{N+1}$, the $\ca$ symbols from the first \({w}\,{\ca^N}\) variable are distinguished, the $\ca$ symbols from the second \({w}\,{\ca^N}\) variable are excluded, and $\beta$ is the string $\to w{\distinguish{\deq}}w$, with $\deq$ distinguished. Then we see that
\begin{equation*}
d=N+M=N+n^{N+1}>n^{N+1}=n^{e+1}.
\end{equation*}
Furthermore, the same argument shows that we may only pump the $\ca$ symbols from ${w}\,{\in}\,{w}\,{\ca^N}$, and any amount of pumping that increases the number of distinguished \(\ca\)s will break the unstratifiability.

\subsection{Introducing the exo-stratification}\label{s:main-thm;ss:exo-stratification}

Let us first consider a somewhat unreasonable syntactic system for first-order formulae. We begin by noting that there are $\aleph_0$ many stratified formulae of set theory and $\aleph_0$ many unstratified formulae of set theory. Therefore, there exists a bijection
\begin{equation*}
f\colon\{\text{formulae of set theory}\}\to\N
\end{equation*}
such that $f(\vphi)$ is even if and only if $\vphi$ is stratified. We define our syntactic system with a single symbol $\star$, saying that $\star^n$ is the formula $f^{-1}(n)$. In this case, the set of stratified formulae is context-free (indeed regular!). However, this syntax is highly undesirable. While one can encode \(f\) in a computable way, it is completely impractical for human use and we run into the issue that the formal language of deduction rules associated with these formulae are not context-free.

We now introduce a new system of syntax of first-order formulae based very closely on the first-order system, which we shall refer to as \emph{enriched syntax}. The only new introduction is a second `dimension' of variables that acts as an in-syntax representation of types. Instead of a variable being $w$ with some number of $\ca$ symbols appended, it is now a $v$ with some number of $\ca$ symbols appended, and some number of $\star$ symbols appended to that. Formally, this is spelled out in the following context-free grammar.

\begin{equation*}
\begin{tabular}{ccccccc}
$\ttF$&$\mapsto$&$\ttA$&$|$&$[\ttF\to\ttF]$&$|$&${[{\forall}\,{\ttV}]}\,{\ttF}$\\
$\ttA$&$\mapsto$&$\bot$&$|$&$\ttV\deq\ttV$&$|$&$\ttV\in\ttV$\\
$\ttV$&$\mapsto$&$\ttP$&$|$&${\ttV}\,{\star}$&\\
$\ttP$&$\mapsto$&$v$&$|$&${\ttP}\,{\ca}$
\end{tabular}
\end{equation*}
Here the symbols \(\ttF\), \(\ttA\), \(\ttV\), and \(\ttP\) can be imagined as standing for `formula', `atomic formula', `variable', and `pre-variable' respectively. Henceforth, for notational convenience, we shall denote the string
\begin{equation*}
v\underbrace{{\ca}\,{\ca}\cdots{\ca}}_{n\text{ many}}\underbrace{{\star}\,{\star}\cdots{\star}}_{m\text{ many}}
\end{equation*}
as $v_n^m$.

On the semantic level, interpreting the strings \(v_n^m\) as variables, there is no difference between the usual syntax and the enriched syntax: Both the collections \(\{{w\ca^n}\mid n\in\N\}\) and \(\{v_n^m\mid n,m\in\N\}\) serve perfectly well at their role of providing an infinite collection of distinct strings that can be isolated from the other components of a formula. Hence one can quite easily translate the definition of, say, a stratified formula over to the enriched syntax, and indeed the proof of Theorem~\ref{thm:main-thm} can be modified (by replacing all \(w\)s with \(v\)s) to show that the language of stratified formulae in the enriched syntax is not context-free. However, this is not the language that we wish to consider; instead, we are only going to pay attention to the \emph{exo-stratified} formulae.

\begin{defn}
A formula $\vphi$ in the enriched syntax system is \emph{exo-stratified} if:
\begin{enumerate}
\item[\textup{1.}] Whenever $v_{n_1}^{m_1}\deq v_{n_2}^{m_2}$ appears as a subformula in $\vphi$, $m_1=m_2$; and
\item[\textup{2.}] whenever $v_{n_1}^{m_1}\in v_{n_2}^{m_2}$ appears as a subformula in $\vphi$, $m_1=m_2-1$.
\end{enumerate}
\end{defn}
Any exo-stratified formula is stratified, e.g.\ by stratification $\sigma(v_n^m)=m$. This is where the name comes from: While most stratified formulae have an implicit stratification, an exo-stratified formula demonstrates its stratification through its variables. However, we impose \emph{no restriction} on the values of $n$ in our variables $v_n^m$. Therefore it is entirely possible for a formula to be exo-stratified, but no longer be stratified if we remove the $\star$ symbols; consider the stratified formula $v_0^1\in v_0^2$ (that is, \({{v}\,{\star}}\in{{v}\,{\star}\,{\star}}\)), which would turn into the unstratified $v\in v$ if we omitted the $\star$ symbols.

Note also that every stratified formula in the enriched syntax may be turned into an exo-stratified formula by substitution of variables. If $\vphi$ is a stratified formula with stratification $\sigma$, and $\pi\colon\N^2\to\N$ is an injection, then replacing each $v_n^m$ by $v_{\pi(n,m)}^{\sigma(v_n^m)}$ produces an exo-stratified formula.

\begin{thm}
The language of exo-stratified formulae is context-free.
\end{thm}

\begin{proof}
The following context-free grammar generates the language of exo-stratified formulae:
\begin{equation*}
\begin{tabular}{ccccccc}
$\ttF$&$\mapsto$&$\ttA$&$|$&$[\ttF\to\ttF]$&$|$&${[{\forall}\,{\ttV}]}\,{\ttF}$\\
$\ttA$&$\mapsto$&$\bot$&$|$&${\ttP}\,{\ttE}$&$|$&${\ttP}\,{\ttM}\,{\star}$\\
$\ttE$&$\mapsto$&$\deq\ttP$&$|$&${\star}\,{\ttE}\,{\star}$&\\
$\ttM$&$\mapsto$&$\in\ttP$&$|$&${\star}\,{\ttM}\,{\star}$&\\
$\ttV$&$\mapsto$&$\ttP$&$|$&${\ttV}\,{\star}$&\\
$\ttP$&$\mapsto$&$v$&$|$&${\ttP}\,{\ca}$&
\end{tabular}
\end{equation*}
This grammar is somewhat more esoteric than the grammars for the usual and enriched syntax systems, the definitions of which lead almost immediately to their grammars. First note that all the old symbols should be interpreted the same way: \(\ttF\) is `formula', \(\ttA\) is `atomic formula', \(\ttV\) is `variable', and \(\ttP\) is `pre-variable'. The new introductions are \(\ttE\) for `equality' and \(\ttM\) for `membership', which concern the $x\deq y$ and $x\in y$ atomic subformulae respectively. We shall describe in detail the process for \(\ttE\), and \(\ttM\) is extremely similar.\\
An atomic formula of the form $x\deq y$ appearing in an exo-stratified formula must be of the form $v_{n_1}^m=v_{n_2}^m$. Hence, the grammar begins by adding a pre-variable to the left of the string. Since exo-stratified formulae have no restriction on the number of $\ca$ symbols appearing on variables, this is safe to do and sets up the \(\ttE\) system. Within \(\ttE\), we want to generate a string of the form ${\star^m}\deq{{\ttP}\,{\star^m}}$, so one either has the choice of adding one star to each side (${\star}\,{\ttE}\,{\star}$) or ending the loop ($\deq\ttP$). Hence, after choosing to add stars $m$ times, the construction looks as follows:
\begin{equation*}
\begin{aligned}
\ttA&\implies{\ttP}\,{\ttE}\\
&\implies{\ttP}\,{\star}\,{\ttE}\,{\star}\\
&\implies{\ttP}\,{\star}\,{\star}\,{\ttE}\,{\star}\,{\star}\\
&\makebox[\widthof{$\implies$}][c]{$\mathrlap{\vdots}$}\\
&\implies{\ttP}\,{\star^m}\,{\ttE}\,{\star^m}\\
&\implies{\ttP}\,{\star^m}\,{\deq}\,{\ttP}\,{\star^m}
\end{aligned}
\end{equation*}
Then it is simply a matter of turning each pre-variable placeholder $\ttP$ into a $v$ with some number of $\ca$ symbols appended.

The case for \(\ttM\) is very similar, but there is one additional $\star$ added to the end of the string on the second variable, ensuring that the formula remains exo-stratified.
\end{proof}

\bibliographystyle{amsplain}
\bibliography{Bibliography/stratifiable_computable.bib}

\section{Acknowledgements}

The results obtained in this paper started off as the product of a University of Cambridge Summer Research in Mathematics (SRIM) Programme in the summer of 2020 with the goal of extending the author's understanding of formal languages and automata. The author would like to thank Thomas Forster for setting up such project despite the limitations of remote learning, for his persistent help throughout that time, for encouraging the writing of this paper, and for the many conversations throughout that have helped to shape it. They also thank Andrew Brooke-Taylor, Randall Holmes, and Asaf Karagila for their feedback on the paper. The author would also like to thank the referee for their feedback on an early version of this paper.

\end{document}